\documentclass[11pt]{article}
\usepackage{latexsym, amssymb,amsthm,amsmath,xypic,lscape,epsfig,setspace}
\usepackage[pdftex]{hyperref}

\theoremstyle{plain}
\newtheorem{theo}{Theorem}[section]
\newtheorem{cor}[theo]{Corollary}
\newtheorem{lem}[theo]{Lemma}
\newtheorem{prop}[theo]{Proposition}

\theoremstyle{definition}
\newtheorem{example}[theo]{Example}
\newtheorem{definition}[theo]{Definition}
\newtheorem{lemma}[theo]{Lemma}

\newenvironment{renumerate}
{
\begin{enumerate}}
{\end{enumerate}
}

\newenvironment{remark}
{\vskip6pt \noindent {\it Remark:}} {\vskip6pt}

{\vskip6pt%
\noindent%
{\it Remarks}: %
\begin{renumerate}}%
{\end{renumerate}\vskip6pt}
\newenvironment{ex}[1]%
{\begin{example}\label{#1}}%
{\end{example}}

\newcommand{\map}{\Psi}

\newcommand{\C}{\text{$\mathbb C$}}

\renewcommand{\frak}[1]{\text{$\mathfrak{#1}$}}
\newcommand{\J}{\text{$\mathcal{J}$}}

\newcommand{\G}{\text{$\mathcal{G}$}}

\newcommand{\ga}{\text{$\alpha$}}

\renewcommand{\gg}{\text{$\gamma$}}
\newcommand{\gh}{\text{$\lambda$}}

\newcommand{\gf}{\text{$\varphi$}}

\newcommand{\Id}{\mathrm{Id}}

\newcommand{\del}{\text{$\partial$}}

\newcommand{\delbar}{\text{$\overline{\partial}$}}
\newcommand{\tensor}{\otimes}

\newcommand{\mc}[1]{\text{$\mathcal{#1}$}}

\newcommand{\into}{\longrightarrow}
\newcommand{\noqed}{\let\qed\relax}

\newcommand{\Gg}{\mathfrak{g}}
\newcommand{\Aa}{\mathfrak{a}}
\newcommand{\Hh}{\mathfrak{h}}

\newcommand{\gcs}{generalized complex structure}

\newcommand{\gcss}{generalized complex structures}

\newcommand{\gk}{generalized K\"ahler}
\newcommand{\gks}{generalized K\"ahler structure}

\newcommand{\wrt}{with respect to}
\renewcommand{\iff}{if and only if}
\newcommand{\Ann}{\mathrm{Ann}}

\newcommand{\Kperp}{\text{$K^{\perp}$}}

\newcommand{\RR}{\mathbb R}

\newcommand{\lra}{\longrightarrow}

\newcommand{\dd}{\mathcal{D}}

\newcommand{\End}{\mathrm{End}}

\newcommand{\IP}[1]{\langle #1\rangle}
\newcommand{\IPP}[1]{\mathrm{tr}( #1)}

\newcommand{\E}{\mathcal{E}}
\newcommand{\ad}{\operatorname{ad}\!}

\newcommand{\Cour}[1]{[\![#1]\!]}

\begin{document}

\title{Reduction of metric structures on Courant algebroids}
\author{Gil R. Cavalcanti \thanks{\tt gilrc@maths.ox.ac.uk}\\ Mathematical Institute,\\ St Giles 24 -- 29\\  Oxford, OX1 3LB, UK.\\[0.1cm]} 
\date{}
\maketitle

\abstract{
We use the procedure of reduction of Courant algebroids introduced in \cite{BCG05} to reduce strong KT, hyper KT and generalized K\"ahler structures on Courant algebroids. This allows us to recover results from the literature as well as explain from a different angle some of the features observed there in.  As an example, we prove that the moduli space of instantons of a bundle over a SKT/HKT/generalized K\"ahler manifold is endowed with the same type of structure as the original manifold.}

\section*{Introduction}

The idea of describing an infinitesimal  group action on a manifold not only by vector fields but also by 1-forms appeared in some instances in the literature. For example, in the context of gauging the Wess--Zumino term of a sigma model \cite{FiSt94} or, in a more geometric way, in the work of Grantcharov {\it et al} \cite{GPP02}, where they study the reduction of strong KT and hyper KT structures. In those cases,  at first,  the 1-forms do not seem to influence the group action, which is just the one generated by the vector fields.
However, the 1-form part of the action has some significance. For example, in the reduction procedure in \cite{GPP02} the strong/hyper KT metric and corresponding complex structures are not modeled on the orthogonal complement of the $G$-orbits, but on a certain space transversal to the group directions which is slanted away from the orthogonal complement by the 1-form part of the action. 

The feature that the model for the tangent space of the quotient is not given by the orthogonal complement to the group directions appeared before in contexts where the presence of a 1-form part of the action was otherwise unnoticed. For example, it appeared in the work of L\"ubke and Teleman \cite{LuTe95}, where they show that the moduli space of instantons of a bundle over an SKT manifold has an SKT structure. More recently this trait resurfaced in the work of Hitchin \cite{Hi05}, where he studies a \gk\ structure on the moduli space of intantons, only there he has to consider two different transversal spaces modeling different structures on the quotient.

In these examples, the need for 1-forms in the action emanates from the presence of a background closed 3-form $H$ --- in \cite{FiSt94} it is the NS-flux, while in the SKT/HKT case it is $d^c\omega$, where $\omega = g(I\cdot, \cdot)$ is the K\"ahler form.
The presence of this closed 3-form suggests that the phenomena above  are related to Courant algebroids.

In this paper we explain the idiosyncrasies pointed out above in terms of extended actions and reduction of Courant algebroids, a procedure introduced in \cite{BCG05} which makes rigorous the idea of an action by vectors and forms and shows how the 1-form part of the action enters as a main ingredient when constructing  the reduced Courant algebroid. In doing so, we also provide a description of the \gk\ structure on the moduli space of instatons found by Hitchin in terms of a \gk\ reduction, answering a question he posed in \cite{Hi05} (this generalized K\"ahler reduction is studied in detail in \cite{BCG07b}).

One of the basic ideas behind this reduction procedure is that one phrases the concept of an infinitesimal $G$-action not only in terms of vector fields, but in terms of an {\it extended action} on the Courant algebroid $TM \oplus T^*M$ generated by vector fields and 1-forms. The vector fields and 1-forms generating the action on the Courant algebroid are required to  be compatible with the underlying infinitesimal action of the vector fields on $M$.

One advancement of this approach is that one can consider the action of a 1-form $\xi$  by itself, without an accompanying vector part. In that case, the compatibility condition  implies that $\xi$ is closed and the reduction procedure consists of taking a maximal leaf $M^{red}$ of the distribution $\Ann(\xi)$ and just restrict the Courant algebroid $TM \oplus T^*M$ to $M^{red}$, obtaining $TM^{red}\oplus T^*M^{red}$. In the general case, an extended action  is given partially by the action of closed forms and partially by a group action, so the effect of the reduction procedure is to take the quotient of a submanifold of $M$ by the action of the group, much in the spirit of symplectic reduction.

In the particular case when the part of the extended action determined by forms alone is given by $d\mu$, with $\mu:M\into \Hh^*$  an equivariant function with values on a $\frak{g}$-module $\Hh^*$, we say that $\mu$ is a {\it moment map} for the action. In this case, reduced manifolds are given by quotients of  inverse images of $G$-orbits in $\Hh^*$ by the action of $G$ on $M$. Of course, the basic motivation for this definition comes from symplectic geometry, as their moment maps are also moment maps in our setting, but our definition is considerably more general. Particular features of our moment map are that it is independent of any geometry on $TM\oplus T^*M$ or on the underlying manifold and also, unlike symplectic moment maps, it is not determined by the group action on the manifold, but it is just part of the extended action. Also, it can have values in any $\frak{g}$-module.

The purpose of this paper is twofold. First we show that our procedure can be used to reduce  SKT and HTK structures, recovering results from \cite{GPP02} and \cite{LuTe95} from a different perspective. From our angle, the reduced structures are always modeled on the orthogonal complement of the group orbits {\it on the Courant algebroid}. Then, we explain the need to look at different spaces, transversal but not orthogonal to the $G$-orbit   {\it on the manifold},  from the fact that the projection $TM \oplus T^*M \into TM$ is not orthogonal.

The second aim of the paper is to provide a concrete example where many of the features of our reduction  procedure are present and, when set up properly, simplify and unify results previously found.  We start  by studying reduction of the strong KT structures on the space of metric connections of a bundle over an SKT manifold.  We write down an explicit  action on the space of connections which is related to the gauge action, but  which involves the action of vectors and forms and has a moment map with values on a $\Gg$-module bigger than $\Gg^*$, namely, $\mu(A) = F_A^+$, the self dual part of the curvature of the connection $A$. Then we  show that the SKT structure can be reduced by this action endowing $\mu^{-1}(0)/G$, i .e., the moduli space of instantons, with an SKT structure. 

Although we focus initially on strong KT reduction, the unifying property of this construction manifests itself on the fact that the computation we do here, in a single sweep,  proves that the moduli space of instantons has a strong KT/hyper KT/generalized K\"ahler structure as long as the starting 4-manifold has the same structure (see also \cite{BCG07} for more details on the \gk\ reduction). There is no need to change the action or the moment map to deal with the different structures. 

This paper is organized as follows. In Section \ref{exact courant algebroids} we introduce exact Courant algebroids. In section \ref{sec:actions} we  review the results from \cite{BCG05} on reduction of Courant algebroids. In Section \ref{sec:SKT} we introduce the generalized metric, strong KT, hyper KT and generalized K\"ahler structures on Courant algebroids. In Section \ref{sec:SKT reduction} we study how to reduce generalized metrics, strong KT and hyper KT and \gk\ structures in the presence of an extended action. There we also give an interpretation of the reduction procedure in terms of the geometry of the tangent bundle, showing that our methods agree with and explain some features of other reduction procedures found in the literature. Then in Section \ref{sec:examples} we work out in detail how to obtain the SKT structure on the moduli space of instantons over an SKT manifold. We then show how that the same method can be used to perform /HKT and \gk\ reduction.

I would like to thank Henrique Bursztyn, Marco Gualtieri, Nigel Hitchin and Oliver Nash for many useful conversations and suggestions. This research is suppported by EPSRC.

\section{Exact Courant algebroids}\label{exact courant algebroids}
A {\it Courant algebroid} over a manifold $M$ is a vector bundle $\E
\to M$ equipped with a fibrewise nondegenerate symmetric bilinear
form $\IP{\cdot,\cdot}$, a bilinear bracket $\Cour{\cdot,\cdot}$ on
the smooth sections $\Gamma(\E)$, and a bundle map $\pi: \E\to TM$, such that for all
$e_1,e_2,e_3\in \Gamma(\E)$ and $f\in C^{\infty}(M)$:
\begin{itemize}
\item[C1)] $\Cour{e_1,\Cour{e_2,e_3}} = \Cour{\Cour{e_1,e_2},e_3} +
\Cour{e_2,\Cour{e_1,e_3}}$,
\item[C2)] $\pi(\Cour{e_1,e_2})=[\pi(e_1),\pi(e_2)]$,
\item[C3)] $\Cour{e_1,fe_2}=f\Cour{e_1,e_2}+ (\pi(e_1)f) e_2$,
\item[C4)] ${\pi(e_1)}\IP{e_2,e_3}= \IP{\Cour{e_1,e_2},e_3}
+ \IP{e_2, \Cour{e_1, e_3}}$,
\item[C5)] $\Cour{e_1,e_1} = \dd\IP{e_1,e_1}$,
\end{itemize}
where, using $\IP{\cdot,\cdot}$ to identify $\E$ with $\E^*$,  we
obtain a bundle map $\pi^* :T^*M \into \E$ dual to $\pi$ and then define $\dd =\tfrac{1}{2}\pi^*\circ d: C^{\infty}(M)\lra \Gamma(\E)$.

Note that C2) and C4) imply that $\pi \circ \pi^*=0$, while $C5)$ implies that the bracket is not skew-symmetric, but rather satisfies
$$
\Cour{e_1,e_2}=-\Cour{e_2,e_1}+2\mathcal{D}\IP{e_1,e_2}.
$$
Since the left adjoint action is a derivation of the bracket (axiom
$C1)$), the pair $(\Gamma(\E),\Cour{\cdot,\cdot})$ is a \emph{Leibniz
algebra}~\cite{Lod93}. 

\begin{definition} A Courant algebroid is {\it exact} if the sequence
\begin{equation}\label{exact}
0 \longrightarrow T^*M \stackrel{\pi^*}{\longrightarrow} \E \stackrel{\pi}{\longrightarrow}
TM \longrightarrow 0
\end{equation}
is exact.
\end{definition}

Given an exact Courant algebroid, we may always choose a right
splitting $\nabla: TM\to \E$ which is \emph{isotropic}, i.e. whose
image in $\E$ is isotropic with respect to $\IP{\cdot,\cdot}$. Such
a splitting has a {\it curvature}: a closed  3-form $H$ defined by
\begin{equation}\label{eq:curv}
H(X,Y,Z) = \IP{\Cour{\nabla(X),\nabla (Y)},\nabla(Z)} \mbox{ for } X,Y,Z\in \Gamma(TM)
\end{equation}
 Using the bundle isomorphism $\nabla+\tfrac{1}{2}\pi^*:TM\oplus T^*M \to \E$,
we transport the Courant algebroid structure onto $TM\oplus T^*M$.
Given $X+\xi,Y+\eta\in \Gamma(TM\oplus TM^*)$, we obtain for the
bilinear pairing
\begin{equation}\label{eq:pairing}
\IP{X+\xi, Y+\eta} = \frac{1}{2}(\eta(X) + \xi(Y)),
\end{equation}
and the bracket becomes
\begin{equation}\label{eq:Hcour}
\Cour{X+\xi, Y+\eta}_H= [X,Y] + \mc{L}_X \eta - i_Y d\xi  + i_Y i_XH,
\end{equation}
called the {\it $H$-Courant bracket} on $TM\oplus T^*M$ \cite{SeWe01} or just the {\it Courant bracket} when the 3-form $H$ is clear from the context.
The bundle map $\pi$ is the projection
\begin{equation}
\pi: TM\oplus T^*M \to TM.
\end{equation}

As observed by \v{S}evera, the choice of a different isotropic splitting of
\eqref{exact} modifies $H$ by an exact 3-form, so the cohomology class $[H]\in H^3(M,\mathbb{R})$
is independent of the splitting and determines the exact Courant algebroid structure on $\E$
completely. We call $[H]$ the  {\it characteristic class} of the Courant algebroid.

\section{Reduction of Courant algebroids}\label{sec:actions}

In this section we review the theory from \cite{BCG05} on extended actions and reduction of Courant algebroids. We also go through some of the examples of Courant algebras and extended actions given in that paper  which will be useful in the next sections.

As we mentioned in Section \ref{exact courant algebroids}, the adjoint action is a symmetry of the bracket and of the pairing and hence represents the natural concept of Lie derivative on an exact Courant algebroid $\mc{E}$. This allows us to regard a section $e$ of $\mc{E}$ as an element of the algebra $\frak{c}$ of  infinitesimal symmetries of $\mc{E}$ therefore giving a map $\ad:\Gamma(\mc{E}) \into \frak{c}$. This map preserves brackets but it is neither injective nor surjective, for example, it is clear from equation \eqref{eq:Hcour} that $\ker(\map) = \Omega^1_{cl}(M)$, the space of closed 1-forms on $M$.

Similarly to the case of a Lie group action, an infinitesimal action on an exact Courant algebroid is described in terms of  sections of $\E$, using the argument above to identify them with symmetries.  However, since the Courant bracket in $\Gamma(\E)$ is not a Lie bracket, the action is not given by a Lie algebra homomorphism, but by a map of  {\it Courant  algebras} (see definition below), an algebraic structure designed to capture the information regarding the Courant bracket and  the pairing on $\E$.

\begin{definition}\label{def:courant algebra}
A \emph{Courant algebra} over the Lie algebra $\Gg$ is a vector
space $\Aa$ equipped with a bilinear bracket
$\Cour{\cdot,\cdot}:\Aa\times\Aa\lra \Aa$ and a map $\pi:\Aa\lra\Gg$,
which satisfy the following conditions for all $\ga_1, \ga_2,
\ga_3\in\Aa$:
\begin{itemize}
\item[$\frak{c}1$)] $\Cour{\ga_1,\Cour{\ga_2,\ga_3}}=\Cour{\Cour{\ga_1,\ga_2},\ga_3}+\Cour{\ga_2,\Cour{\ga_1,\ga_3}}$,
\item[$\frak{c}2$)] $\pi(\Cour{\ga_1,\ga_2})=[\pi(\ga_1),\pi(\ga_2)]$.
\end{itemize}
In other words, $\Aa$ is a Leibniz algebra with a homomorphism to
$\Gg$.

A  Courant algebra is \emph{exact} if $\pi$ is surjective and
$\Cour{\gh_1,\gh_2}=0$ for all $\gh_i\in\ker(\pi)$.
\end{definition}

Of course, these definitions work out so that  if $\E$ is a Courant algebroid, then  $\Gamma(\E)\into \Gamma(TM)$ is a Courant algebra. Also,  $\E$ is an exact Courant algebroid \iff\  $\Gamma(\E)\into \Gamma(TM)$ is an exact Courant algebra. 

If $\Aa\into \Gg$ is an exact Courant algebra, then we automatically  obtain that $\Hh = \ker(\pi)$ is a $\Gg$-module: we define the action by
$$ \gg \cdot \gh = \Cour{\tilde{\gg},\gh},$$ 
where $\tilde{\gg}$ is any element of $\Aa$ such that $\pi(\tilde{\gg}) = \gg$. Since the bracket vanishes on $\frak{h}$ we see that this is a well defined map $\Gg\times \Hh \into \Hh$ and one can easily check that this is indeed an action. A partial converse to this fact is given by the following example.

\begin{example}
[Hemisemidirect product \cite{BCG05,KiWe01}]\label{ex:hemisemi}
Given a $\frak{g}$-module, $\frak{h}$, we can endow $\frak{a} =\frak{g}\oplus \frak{h}$ with the structure of an exact Courant algebra over $\frak{g}$ by taking $\pi:\frak{g} \oplus \frak{h} \into \frak{g}$ to be the  natural projection and  defining
$$\Cour{(\gg_1,\gh_1),(\gg_2,\gh_2)} = ([\gg_1,\gg_2], \gg_1 \cdot \gh_2).$$
This Courant algebra first appeared in \cite{KiWe01}, where it was studied in the context of Leibniz algebras.

\end{example}

A Courant algebra morphism between $\frak{a} \stackrel{\pi}{\into} \frak{g}$ and $\frak{a}' \stackrel{\pi'}{\into} \frak{g}'$ is a commutative diagram formed by bracket preserving maps
$$\xymatrix{\frak{a} \ar[r]^{\pi}\ar[d]^{\Psi} &\frak{g}\ar[d]^{\psi}\\
\frak{a}' \ar[r]^{\pi'}& \frak{g}'}.$$
For exact Courant algebras $\frak{a}$ and $\frak{a}'$, the map $\Psi$ determines the whole diagram. In this case we say that $\Psi$ is a morphism of Courant algebras.

Now, let $\frak{a}$ be an exact Courant algebra and $\E$ an exact Courant algebroid over a manifold $M$. Given morphism of Courant algebras $\map:\frak{a} \into \Gamma(\E)$ we can compose it with $\ad:\Gamma(\E) \into \frak{c}$ to obtain a subalgebra of the algebra of infinitesimal symmetries of $\E$. Also, projecting $\map$ onto $\Gamma(TM)$ we obtain a subalgebra of infinitesimal symmetries of $M$.
$$\xymatrix{\frak{a} \ar[r]^{\map}\ar[d]^{\pi}& \Gamma(\E) \ar[r]^{\ad}\ar[d]^{\pi} & \frak{c}\ar[d]^{\pi}\\
\frak{g} \ar[r]^{\psi} & \Gamma(TM) \ar[r]^{\cong}& \mathrm{dif}(M)}$$
Since the map $\pi:\frak{c} \into \mathrm{dif}(M)$ has a kernel, the algebra of infinitesimal symmetries generated by $\map(\frak{a})$ is, in general, bigger than the corresponding algebra $\psi(\Gg)$. These two will be the same only if $\map(\Hh)$ acts trivially, i. e., $\map(\Hh) \in \Omega_{cl}^1(M)$, where $\Hh$ is the kernel of the projection $\Aa \into \Gg$. If this is the case, it may still happen that the group of symmetries of $\E$ generated by $\map(\Aa)$, $\tilde{G}$, is bigger than the group of symmetries of $M$ generated by $\psi(\Gg)$, $G$, as $\tilde{G}$ may be just a cover of $G$.

\begin{definition}\label{extended action}  
An {\it extended action} is a Courant algebra map $\map:\frak{a}\into \Gamma(\E)$ which generates the same group of symmetries on $\E$ and on $M$. This means that $\map(\Hh) \in \Omega^{1}_{cl}(M)$ and the infinitesimal actions of $\ad\circ \map(\frak{a}/\frak{h})$ and $\psi(\Gg)$ integrate to an action of the same group.
\end{definition}

A practical way to check whether a map of Courant algebras is an extended action  is to choose an isotropic  splitting for $\mc{E}$ making it isomorphic to $(TM \oplus T^*M,\IP{\cdot,\cdot},\Cour{\cdot,\cdot}_H)$ and then requiring that this splitting is preserved, i. e.,
$$ \Cour{\map(\ga), \Gamma(TM)} \subset \Gamma(TM), \qquad \mbox{for all }\ga \in \frak{a}.$$
Letting $\map(\ga) = X_{\ga} + \xi_{\ga}$, this is equivalent to the condition
\begin{equation}\label{H condition}
i_{X_{\ga}}H = d\xi_{\ga},\qquad \mbox{for all } \ga \in \frak{a}.
\end{equation}
If this is the case, then the group action on $TM \oplus T^*M$ is the one induced by its action on  $TM$ and $T^*M$ determined by the underlying diffeomorphisms. Reciprocally, if an extended action induces the action of a compact group on $\E$,  then there is an isotropic splitting preserved by the extended action \cite{BCG05}.

One of the upshots from our point of view is that the concept of moment map appears as an integral part of the action and is independent of any geometry on the underlying manifold or on the Courant algebroid. Indeed, if $\frak{a}\into \frak{g}$ is an exact Courant algebra and $\map:\Aa  \into \Gamma(\E)$ is an extended  action then $\frak{h}$ is a $\frak{g}$-module  and $\map(\gh) \subset \Omega_{cl}^1(M)$ for all $\gh \in \Hh$. Letting  $G$ be the group whose infinitesimal action is determined by $\map$ we see that the $\frak{g}$-action on $\frak{h}$ integrates to a $G$ action.

\begin{definition}
 An extended action $\map$ has a  {\it moment map} if there is a $G$-equivariant map $\mu:M \into \frak{h}^*$ such that for $\gh \in \Hh$,
$$\map(\gh) = d\IP{\mu,\gh}.$$
\end{definition}

Given an extended action,  we have three distributions on $\E$  associated to it: $K = \map(\frak{a})$, $\Kperp$ and $K +\Kperp$ and from these we get three distributions on $M$: $\Delta_b = \pi(K +\Kperp)$, $\Delta_s = \pi(\Kperp)$ and $\pi(K)$, which gives the directions of the group action on $M$.  The distribution $\Delta_s$ can also be described as
$$\Delta_s = \Ann(\map(\frak{h})),$$
and hence, as the annihilator of a space generated by closed forms, it is locally integrable  around the points where $\map(\frak{h})$ has constant rank. Then the leafs of $\Delta_b$ are just the $G$-orbits of the leafs of $\Delta_s$.  If the action admits a moment map $\mu$, then leaves of $\Delta_s$ are level sets of $\mu$ while the leaves of $\Delta_b$ are inverse images of the $G$-orbits. 

If we pick a leaf $P$ of $\Delta_b$ where the group acts freely and properly and over which $\map(\frak{h})$ has constant rank, the reduction procedure can be described using the following facts:
\begin{enumerate}
\item The Courant bracket on $\E$ induces a bracket on $\Gamma(\Kperp|_P)^G$, the space of $G$-invariant sections of the restriction of $\Kperp$ to $P$, which is well defined modulo $\Gamma(K\cap\Kperp|_P)^G$,
\item Using the bracket above in $\Gamma(\Kperp|_P)^G$, we have that $\Gamma(K\cap\Kperp|_P)^G$ is an ideal, so this bracket induces a well defined bracket on
$$\E^{red} = \left.\frac{\Kperp|_P}{(K \cap \Kperp)|_P}\right/G~~~ \mbox{  over }~~~ M^{red}= P/G.$$
\end{enumerate}
This means that given two sections $e_1, e_2 \in \Gamma(\E^{red})$, their bracket is defined by choosing $G$-invariant lifts $\tilde{e_1}, \tilde{e_2} \in \Gamma(\E)|_P$, then extending these lifts {\it arbitrarily} to sections of $\mathtt{e}_1,\mathtt{e}_2 \in \Gamma(\E)$ and then letting
$\Cour{e_1,e_2}$ be section of $\left.\tfrac{\Kperp|_P}{(K \cap \Kperp)|_P}\right/_G$, which has $ \Cour{\mathtt{e}_1,\mathtt{e}_2}|_P \in \Gamma(\Kperp|P)^G$ as a representative.
$$\Cour{e_1,e_2} := \Cour{\mathtt{e}_1,\mathtt{e}_2}|_P + K\cap \Kperp|_P \subset \left.\frac{\Kperp}{K \cap \Kperp}\right|_P,$$
According to (1) and (2) above, this bracket is independent of the choices made. We call $M^{red}$ a {\it reduced manifold} and $\E^{red}$ the {\it reduced Courant algebroid} over it. The argument above is the main ingredient in the proof of the following theorem. 
\begin{theo}
{\em (Burzstyn--Cavalcanti--Gualtieri \cite{BCG05})} Given an extended action $\map$, let $P$ be a smooth leaf of the distribution $\Delta_b$ where $\map(\frak{h})$ has constant rank and on which $G$ acts freely and properly. Then the vector bundle $\mc{E}^{red}$ defined above is a Courant algebroid over $M^{red} = P/G$. $\E^{red}$ is exact \iff\  
$$\pi(K) \cap \pi(\Kperp) = \pi(K \cap \Kperp) \mbox{ over } P.$$
In particular $\E^{red}$ is exact if $K$ is isotropic over $P$.
\end{theo}

We outlined the proof of this theorem because it shows how to concretely obtain the reduced manifold $M^{red}$ and the reduced Courant algebroid $(\mc{E}^{red}, \IP{\cdot,\cdot},\Cour{\cdot,\cdot})$. As we will see later, the condition that $K$ is isotropic over $P$ is an analogue of saying  that $P$ is the inverse image of $0$ by the moment map. 

There are two basic examples one should keep in mind.
\begin{ex}{vectors}
Let  $G$ act freely and properly on $M$ with infinitesimal action $\psi:\Gg\into \Gamma(TM)$ and let $H$ be a basic 3-form. Then $\psi$ gives rise to an extended action of the Courant algebra $\frak{a} = \frak{g}$ on the split Courant algebroid $(TM\oplus T^*M,\IP{\cdot,\cdot},\Cour{\cdot,\cdot}_H)$: $\map(\ga) = \psi(\ga)$.

 In this case $\Kperp = TM \oplus \Ann(\psi(\Gg))$ and hence $\Delta_b = \pi(\Kperp +K) = TM$ has only one leaf: $M$. Therefore the only reduced manifold is $M^{red} = M/G$ and the reduced algebroid is
$$\E^{red} = \frac{\Kperp/G}{K/G} = TM/\psi(\Gg) \oplus \Ann(\psi(\Gg)) \cong TM^{red} \oplus T^*M^{red}.$$

Finally, the basic form $H$ is the curvature of the reduced algebroid for this splitting.
\end{ex}

\begin{ex}{1form}
Consider
the extended action $\rho:\RR\into \Gamma(\E)$ on an exact Courant
algebroid $\E$ over $M$ given by $\rho(1)=\xi$ where $\xi$ is a  closed
1-form, where $\RR \into \{0\}$ is a Courant algebra over the trivial Lie algebra. Then $\Kperp = \{v \in \E :
\pi(v) \in \Ann(\xi)\}$ induces the distribution
$\Delta_b = Ann(\xi)\subset TM$, which is integrable
wherever $\xi$ is nonzero. Since the group action is trivial, a
reduced manifold is simply a choice of integral
submanifold $\iota:P\hookrightarrow M$ for $\xi$ and the reduced Courant algebroid is just the 
pullback exact Courant algebroid. The characteristic  class of $\E^{red}$ in this case is the
pullback to $P$ of the class of $\E$.
\end{ex}

The next example combines the features of the previous two examples and is the basic model for all the  concrete examples we will encounter later in this paper.

\begin{example}\label{ex:momented}
Given a Lie algebra $\Gg$, we can always think of $\Gg$ as a Courant algebra over itself, with the projection given by the identity and the Courant bracket given by the Lie bracket. If $G$ acts on a manifold $M$ with infinitesimal action $\psi:\Gg\into \Gamma(TM)$ and $\E$ is an exact Courant algebroid over $M$, we can always try and {\it lift} this action to an extended action on $\E$:
$$\xymatrix{ \Gg \ar[r]^{\Id}\ar[d]^{\tilde{\map}} & \Gg\ar[d]^{\psi}\\
\Gamma(\E) \ar[r]^{\pi} & \Gamma(TM).}$$
We call  such $\tilde{\map}$ a {\it lifted} or a {\it trivially extended} action. In order for the reduced algebroid to be exact  $K=\tilde{\map}(\Gg)$ must be isotropic, so we will impose that condition on lifted actions.  

If we are also given an equivariant map $\mu:M \into \Hh^*$, where $\Hh$ is a $G$-module, then we can extend the action $\tilde{\map}:\Gg \into \Gamma(\E)$ to an action of $\Aa = \Gg \oplus \Hh$, endowed with the hemisemidirect product structure from Example \ref{ex:hemisemi}, by defining
$$\map(\gg,\gh) = \tilde{\map}(\gg) + d\IP{\mu,\gh}.$$
The equivariance of $\mu$ implies that this is an extended action and $K = \map(\Aa)$ is isotropic over $\mu^{-1}(0)$. The reduced manifolds are $\mu^{-1}(\mc{O}_\gf)/G$, where $\mc{O}_{\gf}$ is the $G$-orbit of $\gf \in \Hh^*$.

In order to describe the reduced algebroid over $\mu^{-1}(0)$ we observe that the reduction can be described in two steps. The first is just restriction to the level set $P= \mu^{-1}(0)$. Then the extended action on $M$ gives rise to an isotropic trivially extended action on $P$ and the second step is to perform the reduction by this action.

If we have an invariant splitting $\E = TM \oplus T^*M$ with the $H$-bracket, i. e., equation \eqref{H condition} holds for this splitting, then the first step gives as intermediate Courant algebroid $TP \oplus T^*P$ with the $\iota^*H$-bracket, where $\iota:P \hookrightarrow M$ is the inclusion map and we have an extended action  $\Gg \into \Gamma(TP \oplus T^*P)$
$$ \gg \mapsto X_{\gg} + \xi_{\gg}.$$
satisfying:
\begin{equation}\tag{\ref{H condition}}
i_{X_{\gg}} H|_{P} = d\xi_{\gg}.
\end{equation}
In particular, projecting onto $T^*P$, we obtain $\xi \in \Omega^1(P;\Gg^*)$.
According to our general assumptions, $G$ acts freely and properly on $P$, making it a principal $G$-bundle. Let $\theta \in \Omega^1(P;\Gg)$ be a connection for this bundle and consider the change of splitting $TP \oplus T^*P$ determined by the 2-form $B = \IP{\theta,\xi}_{\Gg} +\IP{\IP{X,\theta}_{\Gg},\IP{\xi,\theta}_{\Gg}}$,  where $\IP{\cdot,\cdot}$ denotes the pairing between $TM$ and $T^*M$ and $\IP{\cdot,\cdot}_{\Gg}$  the pairing between $\Gg$ and $\Gg^*$. Concretely, if we choose a basis $\gamma_i$ for $\Gg$, with corresponding infinitesimal generators $X_i$  and choose the  dual basis for $\Gg^*$, so that $\theta$ decomposes as a sum of 1-forms $\theta_i$ with $\theta_i(X_j) = \delta_{ij}$, we have
\begin{equation}\label{eq:B-field straighting}
B =  \sum_i \theta_i \xi_i+\frac{1}{2}\sum_{ij}\xi_j(X_i)\theta_i \wedge \theta_j.
\end{equation}

 In this new splitting the infinitesimal generators of the action are given by
\begin{align*}
X_k + \xi_k - i_{X_k}B & = X_k + \xi_k  - \xi_k  + \sum_i \xi_i(X_k)\theta_i - \sum_{j}\xi_j(X_k)\theta_j.\\
=X_k
\end{align*}

Since $B$ is $G$-invariant, the new splitting is also $G$-invariant and now the infinitesimal generators of the action lie in the tangent direction, as in Example \ref{vectors}. , hence the corresponding 3-form curvature of this splitting, $H + dB$ is basic and is the curvature of the induced splitting of the reduced algebroid: $H^{red} = H + dB$.

\end{example}

From this example,  we see that  the requirement that  $K$ is isotropic over $P$ is an analogue of ``$P$ is the inverse image of 0 by the moment map". In the same way symplectic reduction does not rely on taking a particular level set or even the existence of moment maps, the theory we develop in the next sections can be adapted to the case when $K$ is not isotropic over $P$ using the techniques from \cite{BCG05}. However $K$ will be isotropic over $P$ for the applications we have in mind (c.f. Theorems \ref{SKT reduction} and \ref{gk reduction} and Examples  \ref{ex:SKT instantons} and \ref{ex:hitchin}).

We finish this section with an inoffensive remark. If an action has a moment map $\mu:M \into \Hh^*$ and $\Hh$ decomposes as a sum of $\Gg$ modules, $\Hh = \Hh_1 \oplus \Hh_2$, then we can decompose $\mu$ into coordinates $\mu_i:M \into \Hh_i^*$. The reduction procedure to $\mu^{-1}(0)/G$ can then be performed in two steps. The first consisting of just restricting the Courant algebroid and the action to, say, $\mu_1^{-1}(0)$, where we still have an action of $\Hh_2 \into \Aa/\Hh_1 \into \Gg$ and then carrying out a final reduction of $\mu^{-1}_1(0)$ by this action. If $\Hh$ can be decomposed in a different way as a sum of $\Gg$-modules, the final reduced algebroid and reduced manifold  do not depend on the particular decomposition used.

\section {Metric structures on Courant algebroids}\label{sec:SKT}

In this section we introduce the generalized metric and  the concept of  strong KT, hyper KT and \gks s on an exact Courant algebroid. We start with the generalized metric as introduced by Gualtieri \cite{Gu03} and Witt \cite{Wi04}.

\begin{definition} A {\it generalized metric} on an exact  Courant algebroid $\mc{E}^{2n}$ is an orthogonal, self adjoint operator $\mc{G}$ such that $\IP{\mc{G}e,e} >0$ for $e \in \mc{E}\backslash \{0\}$.
\end{definition}

Since $\G$ is orthogonal and self adjoint we see that
$$ \G^2 = \G \G^t = \G \G^{-1} = \Id.$$
So $\G$ splits $\E$ into its $\pm 1$-eigenspaces, which are maximal subbundles where the pairing is ${\pm}$-definite.

The $1$-eigenspace, $V_+$, of $\G$ determines the generalized metric  completely as its orthogonal complement \wrt\ the symmetric pairing is $V_-$, the $-1$-eigenspace of $\G$. Since the natural pairing is definite on $V_+$ and null on $T^*M$ the projection $\pi:V_+ \into TM$ is an isomorphism and hence induces a metric on $TM$:
$$ g(\pi(e),\pi(e)) = \IP{e,e},\qquad e \in V_+.$$
Further, if $\E$ is split as $TM \oplus T^*M$, any such space $V_+$ can be described as the graph over $TM$ of a 2-tensor $g+b$, where $g$ is the metric above and $b$ is a 2-form. Conversely the graph of $g+b$, where  $g$ is a metric  and $b$  a 2-form is a maximal subspace where the pairing is positive and hence $g$ and $b$ determine a generalized metric.

\begin{definition}
Given a generalized metric $\G$ on a Courant algebroid $\E$, the {\it metric splitting} of $\E$ is  $\G(T^*M) \oplus T^*M$.
The {\it curvature of the metric splitting} is the curvature 3-form of this splitting.
\end{definition}
The metric splitting can be characterized by the fact that the metric is the graph of the induced metric $g$ and $b= 0$.

There are two metric structures we want to consider on an exact Courant algebroid: strong KT structures and \gks s. To define the first, recall that a strong KT structure on a manifold $M$ is an integrable complex structure, $I$,  with a hermitian metric, $g$,  for which $dd^{c}\omega =0$, where $\omega(X,Y) = g(I X, Y)$ and $d^c = i(\del-\delbar)$. The correspondent definition in the context of exact Courant algebroids is:

\begin{definition}
A {\it strong KT structure}  on an exact Courant algebroid $\E$ is a generalized metric $\mc{G}$ together with  a complex structure $\mc{I}$ on its $+1$-eigenspace  orthogonal with respect to the symmetric pairing and whose $+i$-eigenspace is closed \wrt\ the bracket.
\end{definition}

\begin{remark}
One should observe that the definitions above for strong KT structures on manifolds and Courant algebroids {\it allow the torsion $H$ (or $d^{c}\omega$) to be zero}, hence for us K\"ahler manifolds are also strong KT.
\end{remark}

Since $\pi:V_+ \into TM$ is an isomorphism which preserves brackets, we can use $\mc{I}$ to define an almost complex structure $I$ on $M$ and the integrability of $\mc{I}$ implies that $I$ is integrable. It is also clear that $I$ is hermitian \wrt\  the  metric $g$ induced on $M$. We claim that $(g,I)$ is a strong KT structure on $M$ for which $[d^c\omega] = -[H]$, where $[H]$ is the characteristic class of the Courant algebroid. 

Before we prove this claim,  we need a lemma. The methods used below follow very closely Gualtieri's argument relating \gk\ and bihermitian structures \cite{Gu03}.

\begin{lem}\label{integrable distribution}
Let $D \subset T_{\C}M$ be a distribution, $c$ a complex 2-form and $H$ be a real closed 3-form. The subbundle of  $(T_{\C}M+T_{\C}^*M,\IP{\cdot,\cdot},\Cour{\cdot,\cdot}_{H})$ defined by
$$ F =\{X + c(X): X \in D\},$$
is integrable \iff\ the subbundle $ D$ is Lie integrable and
$$i_X i_Y (dc - H) = 0 \qquad \forall X, Y \in D.$$
\end{lem}
\begin{proof}
Choose two sections $X +c(X)$ and $Y +c(Y)$ and compute their bracket:
$$\Cour{X +c(X),Y +c(Y)} = [X,Y]  +\mc{L}_X c(Y) - \mc{L}_Y c(X) -d(i_X i_Yc) +i_X i_Y H.$$
If this is to be a section of $F$ we need $[X,Y] \in D$, and hence $D$ is Lie integrable, and 
$$c([X,Y]) = \mc{L}_X c(Y) - \mc{L}_Y c(X) -d(i_X i_Yc) +i_X i_Y H$$
which is only the case if $i_X i_Y(dc - H) =0$.
\end{proof}

\begin{prop}\label{SKT equivalent}The following hold:
\begin{enumerate}
\item Given a strong KT structure $(\mc{G}, \mc{I})$ on an exact Courant algebroid $\E$ over $M$, the induced metric and complex structure make $M$ a strong KT manifold and $[d^c \omega] = -[H]$.
\item Reciprocally, if $(M,g,I)$ is a strong KT structure, the exact Courant algebroid with characteristic class $[-d^{c}\omega]$ admits a strong KT structure on it.
\end{enumerate}
\end{prop}
\begin{proof}
To prove the first claim, choose a splitting of $\E$ so that it is isomorphic to $(TM + T^*M, \IP{\cdot,\cdot}, \Cour{\cdot,\cdot}_{H})$. As we saw, the strong KT structure on $\E$ gives rise to complex structure $I$, a Hermitian metric $g$, and a 2-form $b$ on $M$. Then, the $+i$ eigenspace of $\mc{I}$ is the graph of $b - i \omega$ over $T^{1,0}M$, hence Lemma \ref{integrable distribution} implies that
$$i_X i_Y (db - i d\omega  +H) = 0 \qquad \forall X,Y \in \Gamma(T^{1,0}M).$$
Using the $(p,q)$ decomposition of forms, the equation above is equivalent to
$$-i \del \omega = -\del b^{1,1} - \delbar b^{2,0} - H^{2,1}\qquad \del b^{2,0} = H^{3,0}$$
which once added to its complex conjugate gives us
$$d^{c} \omega = - db - H,$$
proving the first claim.

To prove the second claim we let $H= - d^{c} \omega$ and consider the exact Courant algebroid $(TM\oplus T^*M,\IP{\cdot,\cdot}, \Cour{\cdot,\cdot}_{H})$. Let $\G$ be the generalized metric determined by the maximal positive definite space  $V_+ = \{X + g(X): X \in TM\}$  and consider the complex structure $\mc{I}$  on $V_+$ induced by the complex structure $I$ on $M$  via the isomorphism  $\pi:V_+ \into TM$. Since $T^{1,0}M$ is Lie integrable, according to Lemma \ref{integrable distribution} to prove the integrability of $\mc{I}$  we only have  to check that $i_X i_Y (id\omega -H) =0$ for $X, Y$ in $T^{1,0}M$. But according to the computations in the  first part of the theorem this amounts to $d^c \omega = - H$. 
\end{proof}

\begin{cor}\label{cor:curvature}
If a Courant algebroid has a strong KT structure, then the curvature of the metric splitting is given by $-d^c\omega$.
\end{cor}

Another structure related to strong KT structures are the so called {\it hyper KT structures}, i. e., three complex structures $I,J$ and $K$, with $IJ =-K$ all hermitian \wrt\ a fixed metric $g$ and satisfying
$$ d^c_I\omega_I = d^c_J \omega_J = d^c_K \omega_K  = H \in \Omega^3_{cl}(M),$$
where $\omega_A$ and $d^c_A$ are the K\"ahler form and the $d^c$ operator \wrt\ the complex structure $A$.

Clearly the same arguments from Proposition \ref{SKT equivalent} show that the following is an equivalent  definition of a hyper KT structure.

\begin{definition}
A {\it hyper KT structure}  on an exact Courant algebroid $\E$ is a generalized metric $\mc{G}$ together with  three complex structures $\mc{I}, \mc{J}$ and $\mc{K}$ on its $+1$-eigenspace which are  orthogonal with respect to the symmetric pairing, whose $+i$-eigenspaces are closed \wrt\ the bracket and which satisfy  $\mc{I} \mc{J} = \mc{K}$.
\end{definition}

We finish our study of strong KT structures with a remark. We could have defined a strong KT structure on a Courant algebroid as an integrable orthogonal complex structure on $V_-$, instead of $V_+$. Similarly to the original case, such a structure on a Courant algebroid also induces a strong KT structure on the underlying manifold. The only change in Proposition \ref{SKT equivalent} is that, since $V_-$ is the graph of $b-g$, the torsion of the strong KT structure will be the curvature of the Courant algebroid,
$$ d^c \omega = db +H.$$

The last structure we want to introduce is related to \gcss.

\begin{definition}
A {\it \gcs} on an exact Courant algebroid $\E$ is a complex structure $\J$ on $\E$ orthogonal \wrt\ the pairing and whose $+i$-eigenspace is closed under the bracket.
\end{definition}

\begin{definition}
A {\it \gks} on an exact Courant algebroid $\E$ is a pair of commuting \gcss\ $\J_1$ and $\J_2$ such that $\G = - \J_1 \J_2$ is a generalized metric.
\end{definition}

If $(\J_1,\J_2)$ is a \gks\ on $\E$ and $\G= -\J_1 \J_2$, then the \gcss\ $\J_i$ also commute with $\G$, since
$$ \J_i \G = -\J_i \J_1 \J_2 = - \J_1 \J_2 \J_i = \G \J_i.$$
Therefore $\J_i$ preserves $V_+$ and the $+i$-eigenspace of $\J_i|_{V_+}$ is $L_1 \cap L_2$, the intersection of the $+i$-eigenspaces of $\J_1$ and $\J_2$, which is closed under the bracket. Therefore every \gks\ on $\E$ furnishes a strong KT structure on $\E$, which can be translated into a strong KT  structure $(g,I_+)$ on $M$, according to Proposition \ref{SKT equivalent}.

The same argument used above also shows that $\J_1$ gives a complex structure to $V_-$ and hence a second strong KT structure $(g,I_-)$ on $M$ with the same metric. These two structures are related by
\begin{equation}\label{eq:gk condition}
d_-^c \omega_-  = - d_+^c\omega_+ = H +db,
\end{equation}
where $H$ is curvature of $\E$ for some splitting, $b$ is the 2-form associated to the generalized metric for that splitting and $d^c_{\pm}$ are the usual operators $-i(\del - \delbar)$ computed respectively using the complex structures $I_{\pm}$ on $M$.  Conversely, Gualtieri showed in his thesis \cite{Gu03} that the data above, i.e., a pair of complex structures $I_{\pm}$, a bihermitian metric $g$, a 2-form $b$ and a closed 3-form $H$ satisfying equation \eqref{eq:gk condition}, determines a \gk\ on $(TM \oplus T^*M, \IP{\cdot,\cdot}, \Cour{\cdot,\cdot}_{H})$.

\section{Reduction of metric structures}\label{sec:SKT reduction}

In this section we use the techniques from \cite{BCG05} introduced in Section ~\ref{sec:actions} to reduce metrics and strong KT structures invariant by an extended action.

Throughout this section we let $\map:\frak{a} \into \Gamma(\mc{E})$  be an extended action,  $K$ be the distribution  generated by $\map(\frak{a})$, $P$ be a leaf of the distribution $\Delta_b$ on which $G$ acts freely and properly   and $\mc{E}^{red}$ be the reduced algebroid over  $M^{red} = P/G$. We will assume that $K$ is isotropic over $P$ and hence $\E^{red}$ is exact and given by the quotient
$$ \E^{red} = \left.\frac{K^{\perp G}}{K^G}\right|_{P/G}.$$
We let $p:K^{\perp G} \into \E^{red}$ and $q:P \into M^{red}$ be the natural projections. 

\subsection{Reducing the generalized metric}

A metric $\mc{G}$ on $\mc{E}$ is invariant under the extended action \iff\ its $1$-eigenspace $V_+$ is preserved by the action. In this case, as the pairing on $\E^{red}$ is induced by the pairing on $ K^{\perp G}$, the spaces
$$V_{\pm}^{red}=p(V_{\pm}^G \cap K^{\perp G})$$
are orthogonal  subspaces of $\E^{red}$ and the restriction of the pairing to $V_{\pm}^{red}$ is  $\pm$-definite. They will define a metric as long as they are maximal, i. e.,
$$ V_+^{red} \oplus V_-^{red} = \mc{E}^{red}.$$

\begin{prop}\label{prop:metric}
If the extended action $\map$ preserves $\G$ and $K$ is isotropic over $P$, then $\G$ reduces.
\end{prop}
\begin{proof}
According to the definition of $V_+^{red}$ and $V_-^{red}$, $ V_+^{red} \oplus V_-^{red} = \mc{E}^{red}$ \iff
\begin{equation}\label{eq:metric}
V_+ \cap \Kperp + V_- \cap \Kperp +  K = \Kperp.
\end{equation}
One can easily see that the sum of the first two spaces is $K^{\mc{G}} = \Kperp \cap \mc{G}\Kperp$, i.e.,  the orthogonal complement of $K$ in $\Kperp$. Hence \eqref{eq:metric} holds.
\end{proof}

The proof of the theorem also furnishes a  pictorial description of the reduced metric.  $\mc{G}^{red}$ on $\mc{E}^{red}$.
Let $p:K^{\perp G} \into \mc{E}^{red}$ be the natural projection and $K^{\G} = \Kperp \cap \G \Kperp$ be the metric orthogonal complement of $K$ in $\Kperp$. Then $K^\G$ is $\G$ invariant and $p:K^{\G} \into \E^{red}$ is an isomorphism. The reduced metric is nothing but the image of $\G$ under this isomorphism. If one thinks of $K$ as being the `group directions' on $\E$, then  the reduced metric is modeled in the orthogonal complement of the group directions in $\Kperp$.

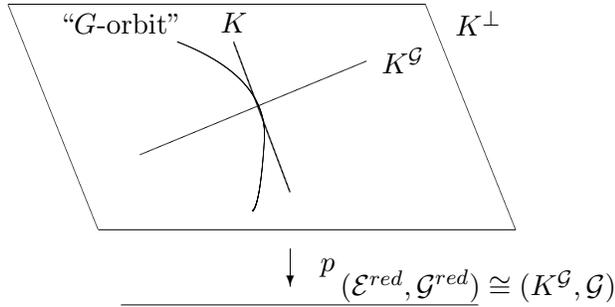
\begin{figure}[h!]
\begin{center}
\unitlength 0.5mm
\begin{picture}(140,90)(0,0)
\linethickness{0.1mm}
\multiput(5,90)(0.12,-0.3){200}{\line(0,-1){0.3}}
\put(5,90){\line(1,0){111}}
\put(29,30){\line(1,0){111}}
\multiput(116,90)(0.12,-0.3){200}{\line(0,-1){0.3}}
\linethickness{0.1mm}
\qbezier(50,80)(74.22,70.62)(73.12,53.75)
\qbezier(73.12,53.75)(72.03,36.88)(70,35)
\linethickness{0.2mm}
\multiput(65,80)(0.12,-0.32){125}{\line(0,-1){0.32}}
\put(35,85){\makebox(0,0)[cc]{``$G$-orbit''}}

\put(65,85){\makebox(0,0)[cc]{$K$}}

\put(110,75){\makebox(0,0)[cc]{$K^{\G}$}}

\put(130,85){\makebox(0,0)[cc]{$K^{\perp}$}}

\linethickness{0.1mm}
\put(80,15){\line(0,1){10}}
\put(80,15){\vector(0,-1){0.12}}
\put(90,20){\makebox(0,0)[cc]{$p$}}

\put(130,15){\makebox(0,0)[cc]{$(\E^{red},\G^{red}) \cong (K^{\G},\G)$}}

\linethickness{0.1mm}
\multiput(40,50)(0.29,0.12){208}{\line(1,0){0.29}}
\linethickness{0.1mm}
\put(35,10){\line(1,0){95}}
\end{picture}
\caption{The generalized metric on the reduced algebroid is modeled on the orthogonal complement of $K$ inside $\Kperp$.}
\label{fig:algebroid viewpoint}
\end{center}
\end{figure}

 \subsection{Metric reduction as seen from $TM$}
 We have seen that a generalized metric $\G$ on an exact algebroid over $M$ gives rise to an actual metric $g$ on $M$. Hence, if an extended action preserves a generalized metric,  we can form the reduced metric $\G^{red}$ on the reduced Courant algebroid $\E^{red}$ and obtain a metric $g^{red}$ on the reduced manifold $M^{red}$.
\begin{equation*}
\xymatrix{(\E,\G)\ar@{~>}[r]^{//}\ar[d]^{\pi}&(\E^{red},\G^{red})\ar[d]^{\pi}\\
(M,g)\ar@{~>}[r]^{//}& (M^{red},g^{red})}
\end{equation*}

As we set out to describe $g^{red}$ in terms of $g$ and the action, we see that there is a price to be paid when one tries to describe generalized objects in terms of the tangent bundle alone. Here we will see that instead of describing $g^{red}$ in terms of the space orthogonal to the $G$-orbits, one has to consider an appropriate space transversal to the $G$-orbits, but slanted according to the action and the 2-form $b$ making up the metric.

 In order to describe the metric $g^{red}$, we fix an invariant splitting $\E = TM \oplus T^*M$, so that the generalized metric is determined by the metric $g$ and a 2-form $b$. Since $\G$ is preserved by the action, one can always find such a splitting, for example, the metric splitting has the properties above. Since we are assuming that $K$ is isotropic over $P$, a leaf of the distribution $\Delta_b$, the extended action on $\E$ gives rise to an isotropic  trivially  extended action on the restricted Courant algebroid over $P$, $\E|_P$, similarly to Example \ref{ex:momented}. Therefore, the action over $P$  is determined by a map $\frak{g} \into \Gamma(TP\oplus T^*P)$, with $ \gamma \mapsto X_{\gamma} + \xi_{\gamma}$.

Recall that $g^{red}$ is just the pushforward  via $\pi$ of the natural pairing on $V^{red}_+$, the $+1$-eigenspace of $\G^{red}$, to $TM^{red}$. But according to the description of $\E^{red}$ from the previous section, we see that $V^{red}_+ = V_+\cap K^{\G} = V_+ \cap \Kperp \cap \G \Kperp$, i.e.,
$$V^{red}_+ = \{Y + (g+b)(Y) \in TP \oplus T^*P: g(Y,X_{\gamma}) +b(Y,X_{\gamma}) + \xi_{\gamma}(Y) =0,~~  \forall \gamma \in \frak{g}\}.$$
This expression shows that $g^{red}$ corresponds to the metric $g$ restricted to a subspace $\tau \subset TP$ transversal to the $G$-orbits which is  slanted from the orthogonal complement of the $G$-orbits by $b$ and $\xi_{\gamma}$ given by
\begin{equation}\label{eq:tau}
 \tau_+ =\{Y \in TP: g(Y,X_{\gamma})= -b(Y,X_{\gamma}) - \xi_{\gamma}(Y),~~  \forall \gamma \in \frak{g}\}
 \end{equation}
\begin{figure}[h!]
\begin{center}

\unitlength 0.5mm
\begin{picture}(140,90)(0,0)
\linethickness{0.1mm}
\multiput(5,90)(0.12,-0.3){200}{\line(0,-1){0.3}}
\put(5,90){\line(1,0){111}}
\put(29,30){\line(1,0){111}}
\multiput(116,90)(0.12,-0.3){200}{\line(0,-1){0.3}}
\linethickness{0.1mm}
\qbezier(50,80)(74.22,70.62)(73.12,53.75)
\qbezier(73.12,53.75)(72.03,36.88)(70,35)
\linethickness{0.1mm}
\multiput(65,80)(0.12,-0.32){125}{\line(0,-1){0.32}}
\put(35,80){\makebox(0,0)[cc]{$G$-orbit}}

\put(65,85){\makebox(0,0)[cc]{$\psi(\frak{g})$}}

\put(112,75){\makebox(0,0)[cc]{$\psi(\frak{g})^{\perp}$}}

\put(130,85){\makebox(0,0)[cc]{$TP$}}

\linethickness{0.1mm}
\put(80,15){\line(0,1){10}}
\put(80,15){\vector(0,-1){0.12}}
\put(90,20){\makebox(0,0)[cc]{$q$}}

\put(130,15){\makebox(0,0)[cc]{$(TM^{red},g^{red}) \cong (\tau,g)$}}

\linethickness{0.1mm}
\multiput(40,50)(0.29,0.12){208}{\line(1,0){0.29}}
\linethickness{0.1mm}
\put(35,10){\line(1,0){95}}
\linethickness{0.1mm}
\multiput(55,45)(0.12,0.14){292}{\line(0,1){0.14}}
\put(95,85){\makebox(0,0)[cc]{$\tau$}}

\end{picture}
\caption{The metric  induced on $TM^{red}$ is modeled on a space transversal to the $G$-orbits but not necessarily orthogonal to them.}
\end{center}
\end{figure}
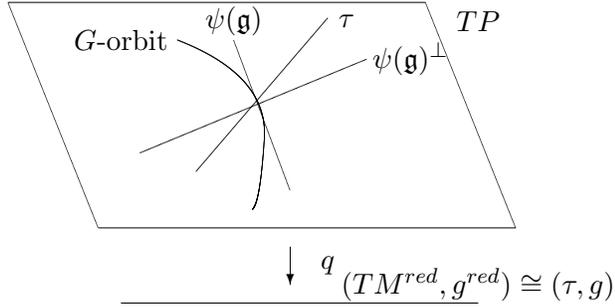

 This feature that the Riemannian model for the reduced manifold is not given by the orthogonal complement of the orbits, but by a different  transversal space was observed before by L\"ubke and Teleman \cite{LuTe95}, Grantcharov, Papadopolous and Poon \cite{GPP02}  and Hitchin \cite{Hi05}. Our approach shows that this puzzling behaviour is nothing but a manifestation on the tangent bundle level of a simpler behaviour of actions on Courant algebroids.

We finish this section on reduction of generalized metrics computing  the curvature of the metric splitting of $\E^{red}$. In what follows we use the metric splitting $\E \cong \G T^*M \oplus T^*M$ to describe the extended action over $P$ as a map $\gg \mapsto X_{\gg} +\xi_{\gg}$, therefore determining a 1-form $\xi \in \Omega^{1}(M,\Gg^*)$. We also use the metric $g$ induced on $P \subset M$ to form a connection $\theta \in \Omega^{1}(P,\Gg)$ on $P$, seen as a principal $G$-bundle over $M^{red} = P/G$, i. e., $\theta(X) =0$ \iff\ $X$ is orthogonal to the $G$-orbit.

\begin{prop}\label{prop:metric curvature} If an extended action $\map$ preserves a generalized metric $\G$, $K$ is isotropic over $P$  and the curvature of the metric splitting of $\E$ over $P$  is $H$, then the curvature of the metric splitting of $\E^{red}$ is 
\begin{equation}\label{eq:severa}
H^{red} =  H + dB = (H + \IP{F,\xi})|_{\tau_+},
\end{equation}
where $\tau_+$ is distribution transversal to the $G$-orbits given in \eqref{eq:tau}, $\theta_+$ is the connection for which $\tau_+$ is the horizontal distribution, $B$ is given by \eqref{eq:B-field straighting} for this connection and $F$ is the curvature of $\theta_+$.
\end{prop}
\begin{proof}
We will start the proof by relating the metric splitting of the reduced algebroid with the subspace $\tau_+$.

We saw that fiberwise $\E^{red} \cong K^\G/G = (\Kperp \cap \G \Kperp|_P)/G$ and the metric on $\E^{red}$ is determined by the metric on $K^\G|_P$. With this identification we have
$$\E^{red} \cong \{Y+ \eta : g(Y,X_\gg) + g(\eta,\xi_{\gg})  =  0 \mbox{ and } \xi_{\gg}(Y) + \eta(X_{\gg}) =0 ~~~~~\forall \gg \in \Gg\}.$$
And $T^*M^{red}$ is identified with $\tau^*: = \ker(\pi:\E^{red} \into TM^{red})$, i. e.,
$$T^*M^{red} \cong \tau^* =\{Y+ \eta : g(Y,X_\gg) + g(\eta,\xi_{\gg})  =  0;~  \xi_{\gg}(Y) + \eta(X_{\gg}) =0\mbox{ and } Y \in \psi(\Gg)  ~~~~~\forall \gg \in \Gg\}.$$
And  the metric splitting of $\E^{red}$ is given by identifying $TM^{red}$ with $\tau= \G(T^*M^{red})$:
$$TM^{red}\cong \tau = \{Y+ \eta : g(Y,X_\gg) + g(\eta,\xi_{\gg})  =  0; ~ \xi_{\gg}(Y) + \eta(X_{\gg}) =0\mbox{ and } g^{-1}\eta \in \psi(\Gg)  ~~~~~\forall \gg \in \Gg\}.$$

The distribution $\tau$ above allows us to define a distribution $\tau_T \subset TP$ transversal to the $G$-orbits, given by $\pi(\tau)$, so that given $Y \in TM_{red}$ there is a unique $Y^h\in TP$, a lift of $Y$, and $\eta \in \Omega^1(P)$ such that $Y^h + \eta\in \tau$. Further, the condition $g^{-1}(\eta) \in \psi(\Gg)$ plus $G$-invariance allows us to define $\chi \in \Omega^1(M^{red};\Gg)$ by the condition that $Y^h+\eta \in \tau$ implies that $\eta = g(\chi(Y))$, where we are identifying $\Gg$ with its image in $\mathcal{X}(P)$. With these conventions we have that
$$\tau = \{Y^h+ g(\chi(Y)):Y \in TM_{red}\}$$
and $Y^h$ and $\chi$ are defined by the conditions
\begin{equation}\label{eq:chi}
g(Y^h,X_\gamma) + \xi_\gamma(\chi(Y)) = g(\chi(Y),X_\gamma) + \xi_\gamma(Y^h) =0 \qquad \forall\gamma \in \Gg.
\end{equation}

Let $\theta$ be the connection on $P$ for which $\tau_T$ is the horizontal space. Then we can describe $\tau_+$ and $\theta_+$ in terms of $\theta$ and $\chi$:

\begin{lemma}\label{lem:theta and theta+} The space $\tau_+$ is given by
$$\tau_+ =\{ Y^h + \chi(Y): Y \in T M^{red}\}$$
and hence $\theta_+ = \theta - \chi$.
\end{lemma}
\begin{proof}
We will prove that $V_+ \cap K^\G $ is the space
\begin{equation}\label{eq:V+red}
v_+ = \{ Y^h + \chi(Y) + g( Y^h + \chi(Y)): Y \in T M^{red}\}
\end{equation}
and since $\tau_+$ is the image of the projection of $V_+\cap K^\G$ to $TP$, the first claim follows. Now $V_+\cap K^\G$ and $v_+$ have the same dimension, so it is enough to prove one inclusion. It is clear that $v_+$ sits inside $V_+$, so we only have to prove that $v_+ \subset \Kperp \cap \G \Kperp$, but since $v_+$ is $G$-invariant, it is enough to prove that $v_+ \subset \Kperp$. For $X_\gamma+\xi_\gamma \in K$, we have
\begin{align*}
\IP{Y^h + \chi(Y) + g( Y^h + \chi(Y)),X_\gamma+\xi_\gamma} &= \xi_{\gamma}(Y^h + \chi(Y)) + g(Y^h + \chi(Y),X_\gamma)\\
&=  \xi_{\gamma}(Y^h) - g(X_\gamma,Y^h) + g(Y^h,X_{\gamma}) -\xi_{\gamma}(Y^h)\\
&=0,
\end{align*}
where we have used \eqref{eq:chi} for the second equality.
\end{proof}

\noindent
{\it Proof of proposition \ref{prop:metric curvature} continued}. Following Example \ref{ex:momented}, we perform an overall $B$-field transform by the  2-form
$$B =  \sum_i \theta_{+i} \xi_i+\frac{1}{2}\sum_{ij}\xi_j(X_i)\theta_{+i} \wedge \theta_{+j},$$
where we have chosen a basis for $\Gg$ and dual basis for $\Gg^*$ and $\theta_{i+}$ are the components of the connection in this basis. The $B$-field transform of $\tau$ is given by
$$e^B\tau =\{Y^h +g(\chi(Y)) - i_{Y^h}B:Y\in TM^{red}\}$$
Since in this splitting $\Kperp = TP \oplus T^*M^{red}$, the 1-form part of a $G$-invariant section of $e^B\tau$ is automatically basic, so to prove that this form vanishes we must prove that when evaluated in any vector in a fixed distribution transverse to the $G$-orbits it results zero. We pick $Z = Z^h + \chi(Z)$ a vector in the horizontal distribution determined by $\theta_+$, then
\begin{align*}
g(\chi(Y),Z) - B(Y^h,Z)& = g(\chi(Y),Z^h+\chi(Z)) - \sum_i\theta_{+i}(Y^h)\xi_i(Z^h + \chi(Z))\\
&= \sum_i \chi_i(Y)(g(X_i,Z^h)+g(X_i,\chi(Z))) +\chi_i(Y)\xi_i(Z^h) + \chi_i(Y)\xi_i(\chi(Z))\\
&= \sum_i \chi_i(Y)(g(X_i,Z^h)+\xi_i(\chi(Z) ) + g(X_i,\chi(Z))+\xi_i(Z^h))\\
&=0
\end{align*}
where $\chi_i$ are the components of $\chi$ in the chosen basis, i.e., $\chi = \sum\chi_i X_i$, in the first equality we used that $\theta_{i+}(Z) =0$, as $Z$ is horizontal, so most terms in $i_Z B$ vanish, in the second equality we expanded $\chi(Y)$ into its components and used Lemma \ref{lem:theta and theta+}  to conclude that $\theta_{+i}(Y^h) = -\chi_i(Y)$ and in the last equality we used \eqref{eq:chi}.

Therefore for the splitting given by the $B$-field transform by $B$, the space $\tau$ agrees with $T M^{red}$, hence this is the metric splitting and the corresponding 3-form $H + dB$ is the curvature of the metric splitting.

Finally, if we restrict this 3-form to the horizontal distribution according to $\theta_+$, we get that $\theta_+|_{\tau_+}=0$, hence
$$H_{red} = (H+dB)|_{\tau_+} = (H + \IP{d\theta_{+},\xi}|_{\tau_+} = (H + \IP{d\theta_+ + \frac{1}{2}[\theta_{+}, \theta_{+}], \xi})|_{\tau_+} = (H + \IP{F,\xi})|_{\tau_+}$$
\end{proof}

\begin{remark}
The same proof holds if one uses $\tau_- = \pi(V_-\cap K^\G) \subset TP$ and $\theta_-$ the connection for which $\tau_-$ is the horizontal distribution. In this case, the $3$-form is still given by equation \eqref{eq:severa}, but now $F$ is the curvature of $\theta_-$.  
\end{remark}

\subsection{Reduction of strong KT structures}

Assume that an extended action $\map:\frak{a} \into \Gamma(\mc{E})$ preserves a strong KT structure $(\mc{G},\mc{I})$ on $\E$. Let $P$ be a leaf of the distribution $\Delta_b$ over which $K$ is isotropic. According to Proposition \ref{prop:metric}, we can reduce the metric. Now we study under which conditions $\mc{I}$ reduces so that $(\G^{red},\mc{I}^{red})$ is a strong KT structure on the reduced algebroid.

Since the metric reduces, we can try and define a reduced  complex structure $\mc{I}^{red}$ on $V_+^{red}$ by giving its $+i$-eigenspace:
\begin{equation}
L^{red} = p(L \cap K^{\perp}_{\C}),
\end{equation}
where $L\subset V_+\tensor \C$ is the $+i$-eigenspace of $\mc{I}$,   $K^{\perp}_{\C} = \Kperp \tensor \C$ and $p:\Kperp \into \E^{red}$ the natural projection. Then $L^{red}$ determines an almost complex structure on $V_+^{red}$ \iff 
\begin{equation}\label{equation}
L^{red} + \overline{L^{red}} = V_+^{red},
\end{equation}
in which case $L^{red} \cap \overline{L^{red}} = \{0\}$ holds trivially, since \eqref{equation} makes  $L^{red}$ a maximal isotropic subspace of $V_+^{red} \tensor \C$.

\begin{theo}\label{SKT reduction}
Let $\map:\frak{a} \into \Gamma(\mc{E})$ be an extended action preserving a strong KT structure $(\mc{G},\mc{I})$ on $\mc{E}$ and let $P$ be a leaf of the distribution $\Delta_b$ over which $K$ is isotropic. Then the  strong KT structure on $\mc{E}$ reduces to a strong KT structure on $\mc{E}^{red}$ \iff\
$$\mc{I}(\Kperp \cap V_+) = \Kperp \cap V_+ \mbox{ over } P.$$
\end{theo}
\begin{proof}
First  we prove that $\mc{I}$ reduces to an almost complex structure on $V_+^{red}$, i. e,
$$ L^{red} \oplus  \overline{L^{red}} = V_+^{red},$$
holds \iff\ $\Kperp \cap V_+$ is invariant under $\mc{I}$. Spelling out, the above is equivalent to
$$\frac{L \cap \Kperp + K}{K} + \frac{\overline{L} \cap \Kperp  + K}{K} = \frac{V_+ \cap \Kperp + K}{K} $$
which is equivalent to
$$ L \cap \Kperp + \overline{L} \cap \Kperp + K= V_+ \cap \Kperp + K.$$
Taking intersection with $V_+$, we see that the above is equivalent to
$$ L \cap \Kperp + \overline{L} \cap \Kperp = V_+ \cap \Kperp .$$
But the space on the left hand side is just $\mc{I}(V_+ \cap \Kperp) \cap V_+ \cap \Kperp$, hence the above is equivalent to
$$ \mc{I}(V_+ \cap \Kperp) \cap (V_+ \cap \Kperp)  = V_+ \cap \Kperp,$$
which is the case \iff\ $\mc{I}(V_+ \cap \Kperp) =  V_+ \cap \Kperp$.

So, to finish the proof  we only have to show that the almost complex structure $\mc{I}^{red}$ defined above is integrable, i. e., $L^{red}$ is closed under the bracket. This follows immediately from the integrability of $L$, since the bracket on $\E^{red}$ is determined by the bracket on $\E$. Indeed, given two sections $e_1, e_2 \in \Gamma(L^{red})$, let $\tilde{e_1}$ and $\tilde{e_2}$ be $G$-invariant lifts of these sections to $\Gamma(L \cap \Kperp|_P)$.  Then we can extend $\tilde{e_1}$ and $\tilde{e_2}$ to $\mathtt{e}_1,\mathtt{e}_2 \in \Gamma(L)$, therefore $\Cour{\mathtt{e}_1,\mathtt{e}_2} \in \Gamma(L)$  and hence
$\Cour{e_1,e_2} = p(\Cour{\mathtt{e}_1,\mathtt{e}_2}|_P) \in L^{red}$.

\end{proof}

\begin{cor}\label{HKT reduction}
Let $\map:\frak{a} \into \Gamma(\mc{E})$ be an extended action preserving a hyper KT structure $(\mc{G},\mc{I},\mc{J},\mc{K})$ on $\mc{E}$ and let $P$ be a leaf of the distribution $\Delta_b$ over which $K$ is isotropic. Then the  hyper KT structure on $\mc{E}$ reduces to a hyper KT structure on $\mc{E}^{red}$ \iff\
$$\mc{I}(\Kperp \cap V_+) = \mc{J}(\Kperp \cap V_+) =\mc{K}(\Kperp \cap V_+) =\Kperp \cap V_+ \mbox{ over P.} $$
\end{cor}

\subsection{Strong KT reduction as seen from $TM$}
We just saw that the metric model for $TM^{red}$ is given by $g$ restricted to the subspace $\tau_+ \subset TP$ defined in \eqref{eq:tau} as  $\tau_+ = \pi(\Kperp \cap V_+)$. Also, the complex structure on $TM$ determined by the strong KT structure on the Courant algebroid is nothing but the push forward of $\mc{I}:V_+ \into V_+$ to $TM$. These two facts together allow us to rephrase Theorem \ref{SKT reduction} as follows.

\begin{cor}
Let $\E$, $\map$ and $P$ be as in Theorem \ref{SKT reduction}. A  strong KT structure $(g,I)$ on $M$ preserved by $\map$ reduces \iff\ $\tau_+$ defined in \eqref{eq:tau} is invariant under $I$.
\end{cor}

So, in this case, not only does $\tau_+$ furnish a Riemannian model for $TM^{red}$, but also the complex structure making it a strong KT manifold. Observe however that from this point of view the integrability of the complex structure on $M^{red}$ is no longer obvious.

We finish remarking that if instead of using $V_+$ to define the strong KT structure on a Courant algebroid we had used $V_-$, all the results in this section would still hold with small adjustments. Most notably, the space modeling the Hermitian structure on $TM^{red}$ would be given by
$$ \tau_- = \{Y \in TP: g(Y,X_{\gamma}) = b(Y,X_{\gamma}) + \xi_{\gamma}(Y),~~  \forall \gamma \in \frak{g}\}$$

\subsection{Reduction of generalized K\"ahler structures}\label{sec:gk instantons}
Reduction of \gks s was studied in \cite{BCG05,BCG07} using the point of view of reduction of \gcss. Their proof of the theorem on \gk\ reduction amounts to proving the following result.

\begin{theo}\label{gk reduction}{\em (Generalized K\"ahler reduction \cite{BCG05,BCG07}):}
Let $\E$, $\map$, and $P$ be as in Theorem~\ref{SKT reduction}, with
$\rho(\frak{a})=K$ isotropic along $P$. If the action preserves a \gk\
structure $(\J_1,\J_2)$ and $\J_iK^\G=K^\G$ along $P$, where $K^\G= \Kperp \cap \G \Kperp$ is the $\G$-orthogonal complement of $K$ in $\Kperp$, then $\J_1$ and
$\J_2$ reduce to a \gks\ on $\E^{red}$.
\end{theo}

The proof of the theorem relies on the fact that $K^\G$ furnishes a model for the reduced Courant algebroid and since $K^\G$ is invariant under $\J_1$ and $\J_2$ both of these structures give rise to \gcss\ on $\E^{red}$. All the pointwise conditions such as $\J_1^{red} \J_2^{red} = \J_2^{red} \J_1^{red}$ and the fact that $-\J_1^{red}\J_2^{red}$ is a metric follow from the same properties of $\J_i$, which still hold restricted to $K^{\G}$. The integrability of the reduced structures follows from the fact that the bracket on $\E^{red}$ is induced by the bracket on $\E$ and the $\J_i$ are integrable. In short, Figure \ref{fig:algebroid viewpoint} is still accurate, as $(K^{\G},\J_1|_{K^{\G}},\J_2|_{K^{\G}})$ gives a model for the \gks\ on $\E^{red}$.

\subsection{Generalized K\"ahler reduction as seen from $TM$}  From the point of view of the tangent bundle, a generalized K\"ahler structure is a bihermitian structure satisfying condition \eqref{eq:gk condition}. So, if an extended action preserves $g,b,I_{\pm}$ and $H$ we can try and perform the strong KT reduction of each of the strong KT structures separately, one using $V_+$ and other using $V_-$. So, in this setting, the condition for reduction is that the spaces $\tau_{\pm}$ defined by
\begin{equation}\label{eq:taupm}
\tau_{\pm} = \{Y \in TP: g(Y,X_{\gamma}) = \mp(b(Y,X_{\gamma}) +\xi_{\gamma}(Y))\}
\end{equation}
 are invariant by $I_{\pm}$, respectivelly.

\begin{figure}[h!]
\begin{center}
\unitlength 0.7mm
\begin{picture}(140,90)(0,0)
\linethickness{0.1mm}
\multiput(5,90)(0.12,-0.3){200}{\line(0,-1){0.3}}
\put(5,90){\line(1,0){111}}
\put(29,30){\line(1,0){111}}
\multiput(116,90)(0.12,-0.3){200}{\line(0,-1){0.3}}
\linethickness{0.1mm}
\qbezier(55,80)(78.91,67.34)(75.62,51.88)
\qbezier(75.62,51.88)(72.34,36.41)(70,35)
\linethickness{0.1mm}
\multiput(65,85)(0.12,-0.3){167}{\line(0,-1){0.3}}
\put(45,80){\makebox(0,0)[cc]{$G$-orbit}}

\put(65,85){\makebox(0,0)[cc]{$\psi(\frak{g})$}}

\put(110,70){\makebox(0,0)[cc]{$\psi(\frak{g})^{\perp}$}}

\put(125,85){\makebox(0,0)[cc]{$TP$}}

\linethickness{0.1mm}
\put(85,15){\line(0,1){10}}
\put(85,15){\vector(0,-1){0.12}}
\put(95,20){\makebox(0,0)[cc]{$q$}}

\put(130,15){\makebox(0,0)[cc]{$(TM^{red},g^{red},I_+^{red})   \cong (\tau_+,g,I_+)$}}
\put(130,5){\makebox(0,0)[cc]{$(TM^{red} ,g^{red},I_-^{red}) \cong (\tau_-,g,I_-)$}}

\linethickness{0.1mm}
\multiput(40,50)(0.36,0.12){167}{\line(1,0){0.36}}
\linethickness{0.3mm}
\put(45,10){\line(1,0){85}}
\linethickness{0.1mm}
\multiput(55,40)(0.12,0.14){292}{\line(0,1){0.14}}
\linethickness{0.1mm}
\multiput(45,70)(0.39,-0.12){167}{\line(1,0){0.39}}
\put(95,80){\makebox(0,0)[cc]{$\tau_+$}}

\put(115,50){\makebox(0,0)[cc]{$\tau_-$}}

\put(120,15){\makebox(0,0)[cc]{}}

\end{picture}
\caption{For the generalized K\"ahler quotient, there are two Riemannian models for $TM^{red}$ given by spaces $\tau_{\pm}$ transversal to the $G$-orbit. Each  of $\tau_{\pm}$ induces a complex structure on $TM^{red}$ making it into a bihermitian manifold.}
\label{fig:gk reduction}
\end{center}
\end{figure}
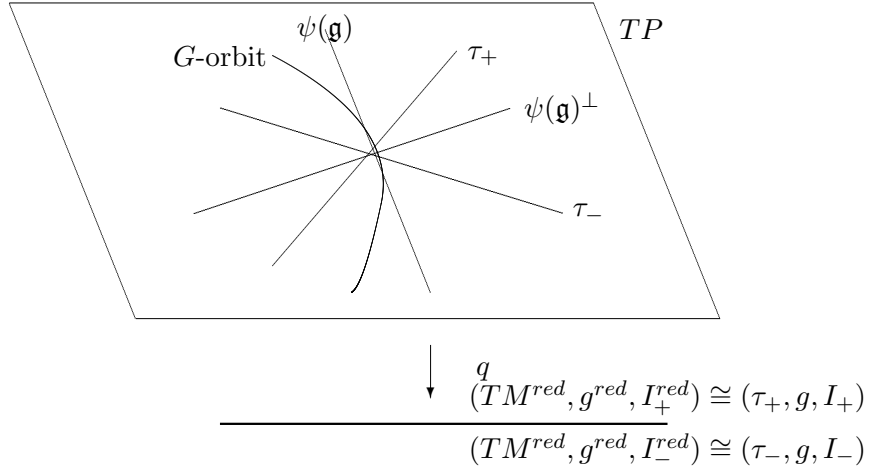

Due to Theorem \ref{gk reduction}, the  reduced strong KT structures pair up to form a \gk\ structure. This can also be seen from fact that the reduced SKT structures are both compatible with the reduced metric, but one is defined on $V_+^{red}$ and the other on $V_-^{red}$, hence, according to Corollary \ref{cor:curvature}
$$d^c_- \omega_- = - d^c_+\omega_+ = H^{red}, $$
where $H^{red}$ is the curvature of the metric splitting and $\omega_{\pm}$ are the K\"ahler forms of the reduced SKT structures.

\begin{cor}
Let $\E$, $\map$ and $P$ be as in Theorem \ref{SKT reduction} with  $K = \map(\frak{a})$ isotropic over $P$. If a \gks\ $(g,b,I_{\pm},H)$ is preserved by the action and the spaces $\tau_{\pm}$ defined by \eqref{eq:taupm} are invariant under $I_{\pm}$, respectively, then the \gk\ structure reduces. 
\end{cor}

\section{Examples}\label{sec:examples}

Now we consider how to use Theorem \ref{SKT reduction} in a sequence of down-to-earth examples.

\begin{example}\label{ex:general KT reduction} Let $(M,g,I)$ be a strong KT manifold and consider the split Courant algebroid $\mc{E} = (TM\oplus  T^*M,\IP{\cdot,\cdot},\Cour{\cdot,\cdot}_H)$, where $H = -d^{c}\omega$. According to Proposition \ref{SKT equivalent}, $(g,I)$ induces a strong  KT structure $(\mc{G},\mc{I})$ on $\mc{E}$, where the 1-eigenspace of $\G$ is the graph of $g$.  Assume we have an extended action $\map:\frak{a} \into \Gamma(TM\oplus T^*M)$ preserving $\mc{G}$, $\mc{I}$ and the splitting, i. e., if $\map(\ga) = X_{\ga} + \xi_{\ga}$ then
\begin{equation}\label{eq:SKT reduction conditions}\begin{cases}
\mc{L}_{X_{\ga}} g =0\\
\mc{L}_{X_{\ga}} I= 0\\
i_{X_{\ga}}H = -d\xi_{\ga}
\end{cases}\qquad \mbox{for all } \ga \in \frak{a}.
\end{equation}

And also assume that $K$ is isotropic over $P$, so the metric reduces to a metric in the reduced algebroid.

For $\ga \in \frak{a}$, define 
$$\tilde{\xi_{\ga}} = g(X_{\ga}) + \xi_{\ga}.$$
Since $V_+ = \{X + g(X)|X \in TM\}$, we have
$$\Kperp \cap V_+ = \{X + g(X)| \tilde{\xi_{\ga}}(X) = 0~~\forall \ga \in \frak{a}\}$$
Therefore the strong KT structure reduces \iff\ 
\begin{equation}\label{dnu and dcnu}
\Ann\{\tilde{\xi_{\ga}}\} \mbox{ is closed under } I.
\end{equation}
Sumarizing, if an extended action satisfies \eqref{eq:SKT reduction conditions} and $K$ is isotropic and \eqref{dnu and dcnu} holds over a leaf $P$ of the distribution $\Delta_b$, then the strong KT structure reduces.
\end{example}

\begin{example}\label{ex:GPP} A particular case of this example was treated by Grantcharov, Papadopoulos and Poon in \cite{GPP02}, where they start with an isotropic trivially extended action
$$
\xymatrix{\frak{g}\ar[r]^{\Id}\ar[d]^{\tilde{\map}}&\frak{g}\ar[d]^{\psi}\\
\Gamma(\E)\ar[r]^{\pi}& \Gamma(TM)}$$
given by
$$\tilde{\map}(\gamma) = X_{\gamma} - g(X_{\gamma}) + d^{c}\IP{\mu,\gamma},$$
where $\mu:M \into \frak{g}^*$ is an equivariant map and such that the action satisfies \eqref{eq:SKT reduction conditions}. Then they use $\mu$ as a moment map for this action, creating in fact an action of the Courant algebra $\frak{a} = \frak{g} \oplus \frak{g}$ with the hemisemidirect product  structure (c.f. Example \ref{ex:hemisemi}) and $\map:\frak{a} \into \E$, as in Example \ref{ex:momented}, given by
\begin{equation}\label{eq:SKT red moment}
\map(\gamma,\lambda) = X_{\gamma} - g(X_{\gamma}) + d^{c}\IP{\mu,\gamma} +d\IP{\mu,\lambda}.
\end{equation}
With these choices, the $\tilde{\xi}_{\alpha}$ in from Example \ref{ex:general KT reduction} are nothing but $d^c\mu$ and $d\mu$, hence they annihilate a space invariant under $I$. Also, according to Example \ref{ex:momented}, $K$ is isotropic over $\mu^{-1}(0)$. Therefore, the quotient of $\mu^{-1}(0)$ by the group action  has a strong KT structure.
\end{example}

In the next example we recover the result of L\"ubke and Teleman stating that the moduli space of instantons over a SKT manifold has an SKT structure on its smooth points \cite{LuTe95}. From our perspective, this example is particularly interesting as it illustrates in a concrete case many different features of the reduction procedure: the extended action is given by vectors and forms together, the moment map is not determined by the geometry on $M$ (but just by the conformal class of the metric $g$) and further, the moment map does not have values on $\Gg^*$, but in a $\Gg$-module.

\begin{example}[SKT structure on the moduli space of instantons]\label{ex:SKT instantons} 
Let $\mc{A}$ be the space of metric connections on a vector bundle $E$ over $M$ and consider the action of the gauge group on $\mc{A}$. 
The Lie algebra of the gauge group is given by $\frak{g} \cong \Omega^0(M,\End(E))$, the tangent space to $\mc{A}$ at a connection $A$ is given by $T_A\mc{A} \cong \Omega^{1}(M,\End(E))$ and the infinitesimal action of the gauge group is given by the map
$$\psi:\frak{g} \into \Gamma(T\mc{A}); \qquad \gamma \stackrel{\psi}{\mapsto} d_A \gamma \in T_A\mc{A}.$$
Our assumption that the action is free and proper and that $\psi$ is an injection means that $d_A$ has no kernel for any connection $A$. These assumptions only hold if the gauge group is simple and no connection renders the bundle $E$ decomposable as a sum of bundles of lower rank with connetions.

The strong KT structure on $M$ induces a strong KT structure on $\mc{A}$ (it is actually a K\"ahler structure) given by following complex structure and Hermitian metric
\begin{equation}\label{eq:complex structure and metric}
\begin{cases}
\mathbb{I} X &= -I^* X,\\
\mathbb{G}(X,Y) &= \int_M\IPP{X \wedge \star Y} = -\int_M\IPP{X \wedge I^*Y}\wedge \omega,
\end{cases}\qquad \forall X,Y \in  \Omega^{1}(M,\End(E))
\end{equation}
where, $\star$ is the Hodge operator,  $\omega = g(I\cdot,\cdot)$ the K\"ahler form associated to $(g,I)$. We identify $T_A^*\mc{A}$ with $\Omega^3(M,\End(E))$, using the integral over $M$ and trace operator. Then we let $\tilde{\map}:\frak{g} \into \Gamma(T\mc{A}\oplus T^*\mc{A})$ be the  trivial extension of the gauge action given by
$$\tilde{\map}(\gamma) = d_A \gamma - H  \gamma.$$
A simple integration by parts shows that this action is isotropic:
\begin{align*}
\IP{\tilde{\map}(\gamma),\tilde{\map}(\gamma)}|_A & =\int_M \IPP{(d_A \gamma) \gamma} \wedge H\\
&=\tfrac{1}{2} \int_M d(\IPP{\gamma}^2) \wedge H =0
\end{align*}

Now, as in Example \ref{ex:momented}, we extend this action using the following equivariant map as a moment map
$$\mu:\mc{A} \into \Omega^2_+(M,\End(E)):=\Hh^*;   \qquad  \mu(A) = F_A^+,$$
i. e., $\mu(A)$ is the self dual part of the curvature of the connetion $A$. Therefore we  obtain an extended action of the hemisemidirect product $\frak{g}\oplus \Hh$ on $T\mc{A} \oplus T^*\mc{A}$.

In the presence of the complex structure $I$, the $\Gg$-module $\Hh^*= \Omega^2_+(M,\End(E))$ decomposes as a sum $\Hh_1 \oplus \Hh_2$, with $\Hh_1^* = \Omega^{2,0+0,2}(M,\End(E))$ and $\Hh_2^*=  \omega\wedge \Omega^0(M,\End(E))$, so we can perform the reduction in two steps. First we restrict to $\mc{A}^{1,1} = \mu_1^{-1}(0)$, i. e., the space of connections with curvature of type $(1,1)$ and then perform a further reduction by the action
$$\map:\Gg\oplus \Hh_2 \into \Gamma(T\mc{A}^{1,1}\oplus T^*\mc{A}^{1,1}),\qquad \map(\gamma,\lambda) =  d_A \gamma - H  \gamma + d\IP{\mu_2,\lambda}.$$

As it is, $\mc{A}^{1,1}$ is a complex submanifold  of $(\mc{A},\mathbb{I})$ and inherits a strong KT structure from $\mc{A}$. So, the final reduced manifold and Courant algebroids are just the reduction of the strong KT structure on $\mc{A}^{1,1}$ by the action $\map$.

We claim that this extended action fits precisely in the setting of Example \ref{ex:GPP}, so that the reduction by the action of $\frak{g}\oplus \frak{g}$ gives a SKT structure to the quotient of $\mu^{-1}(0)$. In order to prove our claim we observe that by wedging with $\omega$ we get the following identification of $\Gg$-modules
$$\Hh_2^* \cong \omega^2 \wedge \Omega^0(M,\End(E)) \cong \Gg^*$$
So, in order to show that we are in same conditions of Example \ref{ex:GPP}, we only need to check that
$$d^c\IPP{\mu_2 \wedge \gamma} - \mathbb{G}(d_A\gamma) = -H \gamma.$$
To prove this we pick an arbitrary $X \in \Omega^1(M,\End(E))$ and show that 
$$\int_Md^c\IPP{\mu_2  \gamma}(X)-\mathbb{G}(d_A\gamma,X) = \int_M \IPP{X  \gamma}H   .$$
We start expanding the left hand side:
\begin{align*}
\int_M d^c\IPP{\mu_2  \gamma}(X) - \mathbb{G}(d_A\gamma,X)& = \mc{L}_{-I^*X}\int_M\IPP{F_A  \gamma}\omega +\int_M\IPP{d_A \gamma \wedge I^*X}\wedge \omega\\
& = \int_M -\IPP{d_A I^*X  \gamma} \omega -\IPP{I^*X \wedge d_A \gamma}\wedge \omega\\
& = \int_M \IPP{I^*X \wedge d_A\gamma} \omega  + \IPP{I^*X  \gamma}d \omega - \IPP{I^*X \wedge d_A \gamma}\wedge \omega\\
& = \int_M \IPP{X  \gamma}\wedge(-I^* d\omega)\\
& = \int_M \IPP{X \gamma}\wedge H,\\
\end{align*}
where in the first equality we used that $d^c = \mathbb{I}^*~d$ and the alternative expression for the metric in \eqref{eq:complex structure and metric}, in the second equality we took the Lie derivative of the first term. The third equality follows by integration by parts of the first term, the fourth from the implicit assumption that $M$ is oriented with the orientation induced by $I$ and the last, from $H = -d^c\omega$.

Therefore we have shown that $\mu^{-1}(0)/G$, the moduli space of instantons, has a strong KT structure. 

We can further compute $d^c\omega$ for this SKT structure using Corolary \ref{cor:curvature} and Proposition \ref{prop:metric curvature}. Since $T\mc{A}\oplus T^*\mc{A}$ in endowed with the standard Courant bracket, the curvature of the metric splitting of the reduced algebroid is given by
$$-d^c\omega = H^{red} = \IP{F,\xi}|_{\tau_+},$$
where $F$ is the curvature of the connection whose horizontal  distribution is $\tau_+$. If the original structure is  K\"ahler, i. e., $H=0$, then $\xi =0$ and the reduced SKT structure obtained above is also K\"ahler.

\end{example}

\begin{example}[HKT structure on the moduli space of instantons]\label{ex:HKT instantons}
The same argument used above, with the same action and moment map $\mu$ shows that if $(M,g,I,J,K)$ is a hyper KT manifold, then $(g,I,J,K)$ induces a hyper KT structure on the moduli space of instantons. The only fact one has to use is that the final reduced structures are independent of the particular way we decompose the $\Gg$-module $\Hh$ and hence the moment map as $\mu = \mu_1 + \mu_2$. If one follows the argument of Example \ref{ex:SKT instantons}, each of $I$, $J$ and $K$ gives a different decomposition of $\frak{h}$ for which it is obvious the corresponding strong KT structure reduces and hence, according to Corollary \ref{HKT  reduction}, they induce an HKT structure on the moduli space of instantons. If $H=0$, and hence $M$ is hyper-K\"ahler, the same argument from Example \ref{ex:SKT instantons} shows that the reduced HKT structure has vanishing torsion and therefore is also hyper-K\"ahler.
\end{example}

\begin{example}[HKT reduction with moment maps]
The standard way to reduce HKT structures found in the literature mimics hyper-K\"ahler reduction and involves a $G$-action together with a moment map with values on $\Gg^*\oplus \Gg^* \oplus \Gg^*$, $\mu = (\mu_I,\mu_J,\mu_K)$ such that 
\begin{enumerate}
\item $d^c_I \mu_I = d^c_J \mu_J =  d^c_K \mu_K$
\item The map $\tilde{\map}:\Gg \into \Gamma(TM\oplus T^*M)$ given by
$$\tilde{\map}(\gg) = \psi(\gg) -g(\psi(\gg)) + d^c_I\mu_I$$
is an isotropic  trivially extended action preserving the splitting $TM \oplus T^*M$ and the HKT structure.
\end{enumerate}
If these conditions hold, we can extend $\tilde{\map}$ to an extended action $\map$ with moment  map  $\mu$, as in Example \ref{ex:momented}. In this case, one can easily see that each strong KT structure reduces, using Example \ref{ex:general KT reduction}. For instance, to check that $(g,I)$ reduces we have to check that the space annihilated by $\{d^c_I\mu_I,d\mu_I,d \mu_J, d \mu_K\}$ is $I$-invariant, but
$$I d\mu_J =  -I J Jd \mu_J = -K d^c_J \mu_J =  -K d^c_K \mu_K = d \mu_K,$$  
showing that $\{d^c_I\mu_I,d\mu_I,d \mu_J, d \mu_K\}$ is $I$-invariant.

Example \ref{ex:HKT instantons} can be seen in the light of an HKT reduction using a moment map with values in $\Gg^*\oplus \Gg^*\oplus \Gg^*$, as done in this example, however by phrasing it this way one might be led to think that the moment map is determined by the full HKT structure, which is not the case, so in this particular case  this approach obscures the picture.
\end{example}

\begin{example}[Generalized K\"ahler structure on the moduli space of instantons]\label{ex:hitchin}
Using the bihermitian point of view, Hitchin proved that the moduli space of instantons over a generalized K\"ahler manifold for which $I_{\pm}$ determine the same orientation has a \gk\ structure \cite{Hi05}. His argument essentially involved two SKT reductions one for each structure, as pictured in Figure \ref{fig:gk reduction} and some work afterwards to prove integrability of the reduced structures.

From our point of view, we can prove this by introducing the following \gks\ on the space of connections $\mc{A}$ on the bundle $E$
\begin{equation}\label{eq:gk structure}
\begin{cases}
\mathbb{I}_+ X &= -I_+^* X,\\
\mathbb{I}_- X &= I_-^* X,\\
\mathbb{G}(X,Y) &= \int_M\IPP{X\wedge \star Y} = -\int_M\IPP{X \wedge  I_{\pm}^*Y}\wedge \omega_{\pm},
\end{cases}\qquad \forall X,Y \in  \Omega^{1}(M,\End(E))
\end{equation}
So that the extended action from Example \ref{ex:SKT instantons} is such that the strong KT structure $(\mathbb{G},\mathbb{I}_+)$ reduces and similar computations also show that $\mathbb{I}_-$ reduces giving rise to a \gk\ pair on $\mu^{-1}(0)$/G.

It is interesting to observe that, similarly to the case of reduction of HKT structures, the decomposition of the moment map $\mu(A) = F_A^+$ in two components used in the argument is different for $I_{\pm}$, however the actual reduction is independent of the choice of intermediate steps.

A direct proof that the \gk\ structure on $\mc{A}$ reduces using Theorem \ref{gk reduction} is given in \cite{BCG07b}
\end{example}

\end{document}